\documentclass[aop,preprint]{imsart}

\usepackage{amsthm,amsmath,natbib}
\RequirePackage[colorlinks]{hyperref}

\startlocaldefs

\usepackage{amsfonts}

\newcommand{\halfplane}{\reals_+\times\reals}
\newcommand{\reals}{\mathbb{R}} 
\newcommand{\rats}{\mathbb{Q}} 
\newcommand{\ints}{\mathbb{Z}}
\newcommand{\E}[1]{\mathbb{E}\left[#1\right]} 
\newcommand{\Et}[1]{\tilde{\mathbb{E}}\left[#1\right]} 
\newcommand{\EP}[2]{\mathbb{E}_{#1}\left[#2\right]}   
\newcommand{\setsF}{\mathcal{F}} 

\newcommand{\nat}{\mathbb{N}}
\newcommand{\Prob}[1]{\mathbb{P}\left(#1\right)} 
\newcommand{\Probt}[1]{\tilde{\mathbb{P}}\left(#1\right)} 
\newcommand{\salg}{$\sigma$-algebra} 
\newcommand{\acd}{ACD}
\newcommand{\support}[1]{{\rm Supp}(#1)}
\newcommand{\msupport}[1]{{\rm MSupp}(#1)}
\newcommand{\slbrack}{{[\hspace*{-0.15em}[}}  
\newcommand{\srbrack}{{]\hspace*{-0.15em}]}}  
\newcommand{\PP}{\mathbb{P}}
\newcommand{\QQ}{\mathbb{Q}}

\newcommand{\D}{{\rm\bf D}}
\newcommand{\borel}{\mathcal{B}}

\newtheorem{definition}{Definition}[section]
\newtheorem{theorem}[definition]{Theorem}
\newtheorem{lemma}[definition]{Lemma}
\newtheorem{corollary}[definition]{Corollary}

\endlocaldefs

\begin{document}

\begin{frontmatter}

\title{Limits of One Dimensional Diffusions}
\runtitle{Limits of Diffusions}

\begin{aug}
\author{\fnms{George} \snm{Lowther}}
\runauthor{George Lowther}
\end{aug}

\begin{abstract}
In this paper we look at the properties of limits of a sequence of real valued inhomogeneous diffusions. When convergence is only in the sense of finite-dimensional distributions then the limit does not have to be a diffusion. However, we show that as long as the drift terms satisfy a Lipschitz condition and the limit is continuous in probability, then it will lie in a class of processes that we refer to as the \emph{almost-continuous diffusions}. These processes are strong Markov and satisfy an `almost-continuity' condition.
We also give a simple condition for the limit to be a continuous diffusion.

These results contrast with the multidimensional case where, as we show with an example, a sequence of two dimensional martingale diffusions can converge to a process that is both discontinuous and non-Markov.
\end{abstract}



\end{frontmatter}

\section{Introduction}

Suppose that we have a sequence of one dimensional diffusions, and that their finite-dimensional distributions converge.
The aim of this paper is to show that, under a Lipschitz condition for the drift components of the diffusions, then the limit will lie in a class of processes that is an extension of the class of diffusions, which we refer to as the \emph{almost-continuous diffusions}. Furthermore, we give a simple condition on this limit in order for it to be a continuous diffusion.

One way that an inhomogeneous diffusion can be defined is by an SDE
\begin{equation}\label{eqn:SDE}
dX_t = \sigma(t,X_t)\,dW_t + b(t,X_t)\,dt
\end{equation}
where $W$ is a Brownian motion. Under certain conditions on $\sigma$ and $b$, such as Lipschitz continuity, then it is well known that this SDE will have a unique solution (see \citep{Protter} V.3, \citep{Revuz} IX.2, \citep{Rogers} V.11). Furthermore, whenever the solution is unique then $X$ will be a strong Markov process (see \citep{Protter} V.6, \citep{Rogers} V.21). More generally, we can consider all possible real valued and continuous strong Markov processes.

We now ask the question, if we have a sequence $X^n$ of such processes whose finite-dimensional distributions converge, then does the limit have to be a continuous and strong Markov process?
In general, the answer is no. There is no reason that the limit should either be continuous or be strong Markov.
In the case of tight sequences (under the topology of locally uniform convergence) then convergence of the finite-dimensional distributions is enough to guarantee convergence under the weak topology, and the limits of continuous processes under the weak topology are themselves continuous (see \citep{Revuz} Chapter XIII or \citep{HeWangYan} Chapter 15).

However, we shall look at the case where the finite-dimensional distributions converge, but do not place any tightness conditions on the processes. In fact, we shall only place a Lipschitz condition on the increasing part of $b(t,x)$ (w.r.t. $x$) for processes given by the SDE (\ref{eqn:SDE}), and place no conditions at all on $\sigma(t,x)$. We further generalize to processes that do not necessarily satisfy an SDE such as (\ref{eqn:SDE}), but only have to satisfy the strong Markov property and a continuity condition.

In this case there is no need for the limit of continuous processes to be continuous, as we shall see later in a simple example. However, in the main result of this paper, we show that as long as the limit is continuous in probability, then it will be strong Markov and satisfy a pathwise continuity condition --- which we shall refer to as being \emph{almost-continuous}. Furthermore, under simple conditions on the limit, then it can be shown to be a continuous process.

The extension of continuous one dimensional diffusions that we require is given by the \emph{almost-continuous diffusions} that we originally defined in \citep{Lowther1}. 

\begin{definition}\label{defn:acd}
Let $X$ be a real valued stochastic process. Then,
\begin{enumerate}
\item $X$ is \emph{strong Markov} if for every bounded, measurable $g:\reals\rightarrow\reals$ and every $t\in\reals_+$ there exists a measurable $f:\halfplane\rightarrow\reals$ such that
\begin{displaymath}
f(\tau,X_\tau) = \E{g(X_{\tau+t})\mid \setsF_\tau}
\end{displaymath}
for every finite stopping time $\tau$.

\item $X$ is almost-continuous if it is cadlag, continuous in probability and given any two independent, identically distributed cadlag processes $Y,Z$ with the same distribution as $X$ and for every $s<t\in\reals_+$ we have
\begin{equation*}
\Prob{Y_s<Z_s,\,Y_t>Z_t{\rm\ and\ }Y_u\not= Z_u\textrm{ for every }u\in(s,t)}=0
\end{equation*}

\item $X$ is an \emph{almost-continuous diffusion} if it is strong Markov and almost-continuous.
\end{enumerate}
\end{definition}

We shall often abbreviate `almost-continuous diffusion' to \acd.
Note that the almost-continuous property simply means that $Y-Z$ cannot change sign without passing through zero, which is clearly a property of continuous processes.
In \citep{Lowther1} we applied coupling methods to prove that conditional expectations of functions of such processes satisfy particularly nice properties, such as conserving monotonicity and, in the martingale case, Lipschitz continuity and convexity. These methods were originally used by \citep{Hobson} in the case of diffusions that are a unique solution to the SDE (\ref{eqn:SDE}). As the results in this paper show, almost-continuous diffusions arise naturally as limits of continuous diffusions, and our method of proof will also employ similar coupling methods.
Furthermore, in a future paper we shall show that, subject to a Lipschitz constraint on the drift component, any almost-continuous diffusion is a limit of continuous diffusions (under the topology of convergence of finite-dimensional distributions).

We now recall that the weak topology on the probability measures on $\left(\reals^d,\borel(\reals^d)\right)$ is the smallest topology making the map $\mu\mapsto \mu(f)$ continuous for every bounded and continuous $f:\reals^d\rightarrow\reals$.
In particular, a sequence $(\mu_n)_{n\in\nat}$ of probability measures on $\reals^n$ converges weakly to a measure $\mu$ if and only if
\begin{equation*}
\mu_n(f)\rightarrow\mu(f)
\end{equation*}
for every bounded and continuous $f:\reals^d\rightarrow\reals$.

Now, suppose that we have real valued stochastic processes $(X^n)_{n\in\nat}$ and $X$, possibly defined on different probability spaces. Then, for any subset $S$ of $\reals_+$, we shall say that $X^n$ converges to $X$ in the sense of finite-dimensional distributions on $S$ if and only if for every finite subset $\{t_1,t_2,\ldots,t_d\}$ of $S$ then the distributions of $(X^n_{t_1},X^n_{t_2},\ldots,X^n_{t_d})$ converges weakly to the distribution of $(X_{t_1},X_{t_2},\ldots,X_{t_d})$.

We shall use the space of cadlag real valued processes (Skorokhod space) on which to represent the probability measures, and use $X$ to represent the coordinate process.
\begin{align*}
&\D=\left\{\textrm{cadlag functions }\omega:\reals_+\rightarrow\reals\right\},\\
&X:\reals_+\times\D\rightarrow\reals,\ (t,\omega)\mapsto X_t(\omega)\equiv \omega(t),\\
&\setsF=\sigma\left(X_t:t\in\reals_+\right),\\
&\setsF_t=\sigma\left(X_s:s\in[0,t]\right).
\end{align*}
Then, $(\D,\setsF)$ is a measurable space and $X$ is a cadlag process adapted to the filtration $\setsF_t$.

With these definitions, a sequence $\PP_n$ of probability measures on $(\D,\setsF)$ converges to $\PP$ in the sense of finite-dimensional distributions on a set $S\subseteq\reals_+$ if and only if
\begin{equation*}
\EP{\PP_n}{f(X_{t_1},X_{t_2},\ldots,X_{t_d})}\rightarrow\EP{\PP}{f(X_{t_1},X_{t_2},\ldots,X_{t_d})}
\end{equation*}
as $n\rightarrow\infty$ for every finite $\{t_1,t_2,\ldots,t_d\}\subseteq S$ and every continuous and bounded $f:\reals^d\rightarrow\reals$.

We now state the main result of this paper.

\begin{theorem}\label{thm:limit of acd is acd}
Let $(\PP_n)_{n\in\nat}$ be a sequence of probability measures on $(\D,\setsF)$ under which $X$ is an almost-continuous diffusion. Suppose that there exists a $K\in\reals$ such that, for every $n\in\nat$, the process $X$ decomposes as
\begin{equation*}
X_t=M^n_t+\int_0^t b_n(s,X_s)\,ds
\end{equation*}
where $M^n$ is an $\setsF_\cdot$-local martingale under $\PP_n$ and $b_n:\halfplane\rightarrow\reals$ is locally integrable and satisfies
\begin{equation*}
b_n(t,y)-b_n(t,x)\le K(y-x)
\end{equation*}
for every $x<y\in\reals$ and every $t\in\reals_+$.

If $\PP_n\rightarrow\PP$ in the sense of finite-dimensional distributions on a dense subset of $\reals_+$, and $X$ is continuous in probability under $\PP$, then it is an almost-continuous diffusion under $\PP$.
\end{theorem}

The proof of this result will be left until Sections \ref{sec:ac} and \ref{sec:sm}.
Note that in the special case where $K=0$ then the condition simply says that $b_n(t,x)$ is decreasing in $x$. Furthermore, Theorem \ref{thm:limit of acd is acd} reduces to the following simple statement in the martingale case.

\begin{corollary}\label{cor:limit of acd mart is acd}
Let $(\PP_n)_{n\in\nat}$ be a sequence of probability measures on $(\D,\setsF)$ under which $X$ is an \acd\ martingale. If $\PP_n\rightarrow\PP$ in the sense of finite-dimensional distributions on a dense subset of $\reals_+$ and $X$ is continuous in probability under $\PP$, then it is an almost-continuous diffusion under $\PP$.
\end{corollary}

We can also give a simple condition on the measure $\PP$ from Theorem \ref{thm:limit of acd is acd} and Corollary \ref{cor:limit of acd mart is acd} in order for $X$ to be a continuous process. Recall that the support of the real valued random variable $X_t$ is the smallest closed subset $C$ of the real numbers such that $\Prob{X_t\in C}=1$.

\begin{lemma}\label{lemma:ac connected supp is cont}
Let $X$ be an almost-continuous process. If the support of $X_t$ is connected for every $t$ in $\reals_+$ outside of a countable set then $X$ is continuous.
\end{lemma}

The proof of this result is left until the end of Section \ref{sec:ac}, and follows quite easily from the properties of the \emph{marginal support} of a process, which we studied in \citep{Lowther1}.

The results above (Theorem \ref{thm:limit of acd is acd} and Corollary \ref{cor:limit of acd mart is acd}) are, in a sense, best possible. Certainly, it is possible for a continuous and Markov (but non-strong Markov) martingale to converge to a process that is neither almost-continuous nor Markov. Similarly, a strong Markov but discontinuous martingale can converge to a process that is not Markov.
Furthermore, these results do not extend in any obvious way to multidimensional diffusions --- in Section \ref{sec:examples} we shall construct an example of a sequence of continuous martingale diffusions taking values in $\reals^2$, and which converge to a discontinuous and non-Markov process.

Now, suppose that we have any sequence of probability measures $\PP_n$ on $(\D,\setsF)$ under which $X$ is an almost-continuous diffusion.
In order to apply Theorem \ref{thm:limit of acd is acd} we would need to be able to pass to a subsequence whose finite-dimensional distributions converge.
It is well known that if the sequence is tight with repect to the Skorokhod topology, then it is possible to pass to a subsequence that converges weakly with respect to this topology (see \citep{Revuz} Chapter XIII or \citep{HeWangYan} Chapter 15). We do not want to restrict ourselves to this situation. Fortunately, it turns out that under fairly weak conditions on $X$ then it is possible to pass to a subsequence that converges in the sense of finite-dimensional distributions. This follows from the results in \citep{Meyer}, where they consider convergence under a topology that is much weaker than the Skorokhod topology, but is still strong enough to give convergence of the finite-dimensional distrubutions in an almost-everywhere sense.

By `convergence almost everywhere' in the statement of the result below, we mean that there is an $S\subseteq\reals_+$ such that $\reals_+\setminus S$ has zero Lebesgue measure and $\PP_{n_k}\rightarrow\PP$ in the sense of finite-dimensional distributions on $S$. In particular, $S$ must be a dense subset of $\reals_+$. Recall that we are working under the natural filtration $\setsF_\cdot$ on Skorokhod space $(D,\setsF)$.

\begin{theorem}\label{thm:ae convergent subsequence}
Let $(\PP_n)_{n\in\nat}$ be a sequence of probability measures on $(\D,\setsF)$ under which $X$ has the decomposition
\begin{equation*}
X=M^n+A^n,
\end{equation*}
where $M^n$ is a cadlag $\PP_n$-martingale and $A^n$ is an adapted process with locally finite variation. Suppose further that for every $t\in\reals_+$ the sequence
\begin{equation*}
\EP{\PP_n}{|X_t|}+\EP{\PP_n}{\int_0^t\,|dA^n_s|}
\end{equation*}
is finite and bounded.

Then, there exists a subsequence $(\PP_{n_k})_{k\in\nat}$ and a measure $\PP$ on $(\D,\setsF)$ such that $\PP_{n_k}\rightarrow\PP$ in the sense of finite-dimensional distributions almost everywhere on $\reals_+$.
\end{theorem}
\begin{proof}
We use the results from \citep{Meyer} for tightness under the pseudo-path topology of a sequence of processes with bounded conditional variation.

For every $k\in\nat$ define the process $Y^k_t \equiv 1_{\{t<k\}}X_t$. Then the conditional variation of $Y^k$ under the measure $\PP_n$ satisfies
\begin{equation*}
V_n(Y^k)\le \EP{\PP_n}{|X_k|}+\EP{\PP_n}{\int_0^k\,|dA^n_s|}
\end{equation*}
which is bounded over all $n\in\nat$ by some constant $L_k$. Now define the process
\begin{equation*}
Z_t = \sum_{k=1}^\infty 2^{-k}\left(L_k+1\right)^{-1}Y^k_t= \theta(t)X_t
\end{equation*}
where $\theta$ is the cadlag function
\begin{equation*}
\theta(t) = \sum_{k=1}^\infty 2^{-k}\left(L_k+1\right)^{-1}1_{\{t<k\}}.
\end{equation*}
Then, the conditional variation of $Z$ satisfies
\begin{equation*}
V_n(Z)\le\sum_{k=1}^\infty 2^{-k}\left(L_k+1\right)^{-1}V_n(Y^k)\le\sum_{k=1}^\infty 2^{-k}=1.
\end{equation*}
So, by Theorem 4 of \citep{Meyer}, there exists a subsequence $(\PP_{n_k})_{k\in\nat}$ under which the laws of the process $Z$ converge weakly (w.r.t. the pseudo-path topology) to the law of $Z$ under a probability measure $\PP$.
Then by Theorem 5 of \citep{Meyer}, we can pass to a further subsequence such that the finite-dimensional distributions of $Z$ converge almost everywhere to those under $\PP$. Finally, as $X_t=\theta(t)^{-1}Z_t$, we see that the finite-dimensional distributions of $X$ also converge almost everywhere.
\end{proof}

These results give us a general technique that can be used to construct almost-continuous diffusions whose finite-dimensional distributions satisfy a desired property. That is, we first construct a sequence of almost-continuous diffusions whose distributions satisfy the required property in the limit. Then, we can appeal to Theorem \ref{thm:ae convergent subsequence} in order to pass to a convergent subsequence and use Theorem \ref{thm:limit of acd is acd} to show that the limit is an almost-continuous diffusion.
This is a method that we shall use in a later paper in order to construct \acd\ martingales with prescribed marginal distributions.

\section{Examples}
\label{sec:examples}

We give examples demonstrating how the convergence described  in Theorem \ref{thm:limit of acd is acd} behaves, and in particular show how a continuous diffusion can converge to a discontinuous process satisfying the almost-continuity condition.

We then give an example showing that Theorem \ref{thm:limit of acd is acd} and Corollary \ref{cor:limit of acd mart is acd} do not extend to multidimensional diffusions.

\subsection{Convergence to a reflecting Brownian motion}
We construct a simple example of continuous martingale diffusions converging to a reflecting Brownian motion. Consider the SDE
\begin{align}
&dX^n_t = \sigma(X^n_t)\,dW_t,\label{eqn:refl bm}\\
&\sigma_n(x) = \max(1,-nx)\nonumber
\end{align}
for each $n\in\nat$, with $X^n_0=0$. Here, $W$ is a standard Brownian motion. As $\sigma_n$ are Lipschitz continuous functions, these SDEs have a unique solution and $X^n$ will be strong Markov martingales. In particular, they will be almost-continuous diffusions. We shall show that they converge to a reflecting Brownian motion.

The SDE (\ref{eqn:refl bm}) can be solved by a time change method, where we first choose any Brownian motion $B$ and define the processes
\begin{align}
&A^n_t = \int_0^t\sigma_n(B_s)^{-2}\,ds,\label{eqn:time change defs}\\
&T^n_t = \inf\left\{T\in\reals_+:A^n_T>t\right\}.\nonumber
\end{align}
Then the process
\begin{equation*}
X^n_t=B_{T^n_t}
\end{equation*}
gives a weak solution to SDE (\ref{eqn:refl bm}). We can take limits as $n\rightarrow\infty$,
\begin{equation*}
A^n_t\rightarrow A_t\equiv\int_0^t1_{\left\{B_s\ge 0\right\}}\,ds.
\end{equation*}
If we now use $A$ to define the time change,
\begin{align}
&T_t = \inf\left\{T\in\reals_+:A_T>t\right\},\label{eqn:X lim time change defn}\\
&X_t = B_{T_t}\nonumber
\end{align}
then $X$ is a Brownian motion with the negative excursions removed, and so is a reflecting Brownian motion.

For every $t\in\reals_+$ we have $X_t>0$ (a.s.) and so $A$ is strictly increasing in a neighbourhood of $t$. Therefore, $T^n_t\rightarrow T_t$. This shows that $X^n_t\rightarrow X_t$ (a.s.), so the processes $X^n$ do indeed converge to $X$ in the sense of finite-dimensional distributions.
However, it does not converge weakly with respect to the topology of locally uniform convergence. In fact, the minimum of $X^n$ over any interval does not converge weakly to the minimum of $X$.
\begin{equation*}
\inf_{s\le t}X^n_s=\inf_{s\le T^n_t}B_s\rightarrow\inf_{s\le T_t} B_s<0 = \inf_{s\le t}X_s
\end{equation*}
for every $t>0$. This example shows that a limit of martingale diffusions need not be a martingale. However, note that the support of $X_t$ is $[0,\infty)$ for any positive time $t$, and $X$ has no drift over any interval that it does not hit $0$. This is true more generally --- whenever a process is a limit of one dimensional martingale diffusions, then it will behave like a local martingale except when it hits the edge of its support.

\subsection{Convergence to a symmetric Poisson process}
We show how continuous diffusions can converge to a discontinuous process, such as the symmetric Poisson process. By `symmetric Poisson process' with rate $\lambda$, we mean a process with independent increments whose jumps occur according to a standard Poisson process with rate $\lambda$ and such that the jump sizes are independent and take the values $1$ and $-1$, with positive and negative jumps equally likely. Alternatively, it is the difference of two independent Poisson processes with rate $\lambda/2$.

If $X$ is a symmetric Poisson process with $X_0=0$, then it follows that the support of $X_t$ is $\ints$ for every positive time $t$ and it is easy to show that it satisfies the almost-continuous property.

We now let $\sigma_n:\reals\rightarrow\reals$ be positive Lipschitz continuous functions such that $\sigma_n^{-2}$ converges to a sum of delta functions at each integer point. For example, set
\begin{equation}\label{eqn:sn defn for approx poisson}
\sigma_n(x) = \left(\pi/n\right)^{1/4}\left(\sum_{k=-\infty}^\infty\exp(-n(x+k)^2)\right)^{-1/2}
\end{equation}
In particular, this gives
\begin{equation}\label{eqn:s-2 goes to deltas}
\int f(x)\sigma_n(x)^{-2}\,dx\rightarrow\sum_{k\in\ints}f(k)
\end{equation}
as $n\rightarrow\infty$, for all continuous functions $f$ with compact support.
We now consider the SDE
\begin{equation}\label{eqn:sde approx poisson}
dX^n_t=\sigma_n(X^n_t)\,dW_t
\end{equation}
where $W$ is a standard Brownian motion, and $X^n_0=0$. As $\sigma_n$ is Lipschitz continuous, $X^n$ will be an \acd\ martingale. We can solve this SDE by using a time changed Brownian motion, in the same way as for the previous example.
So, let $B$ be a standard Brownian motion and $A^n_t$, $T^n_t$ be defined by equations (\ref{eqn:time change defs}). Then $X^n_t=B_{T^n_t}$ solves SDE (\ref{eqn:sde approx poisson}).

If we let $L^a_t$ be the semimartingale local time of $B$ at $a$, then it is jointly continuous in $t$ and $a$ and Tanaka's formula gives
\begin{equation*}
A^n_t = \int L^a_t \sigma_n(a)^{-2}\,da.
\end{equation*}
Equation (\ref{eqn:s-2 goes to deltas}) allows us to take the limit as $n$ goes to infinity,
\begin{equation*}
A^n_t\rightarrow A_t\equiv\sum_{a\in\ints}L^a_t.
\end{equation*}
Then, $A$ will be constant over any time interval for which $B\not\in\ints$ and it follows that if we define $T_t$ and the time changed process $X$ by equations (\ref{eqn:X lim time change defn}) then the support of $X_t$ will be contained in $\ints$ for every time $t$. In fact, $X$ will be a symmetric Poisson process.

As in the previous example, we have $T^n_t\rightarrow T_t$ as $n\rightarrow\infty$ (a.s.). Therefore $X^n_t\rightarrow X_t$ (a.s.) for every $t\in\reals_+$, showing that the continuous martingale diffusions converge to the discontinuous process $X$.

\subsection{A discontinuous and non-Markov limit of multidimensional martingale diffusions}
We give an example of a sequence of $2$-dimensional continuous diffusions converging to a discontinuous and non-Markov process. This shows that Theorem \ref{thm:limit of acd is acd} and Corollary \ref{cor:limit of acd mart is acd} do not extend to the multidimensional case in any obvious way.
To construct our example, first define the Lipschitz continuous function $f:\reals\rightarrow\reals$ by
\begin{equation*}
f(x) = \min\left\{|x-k|:k\in\ints\right\}.
\end{equation*}
Now let $U$ be a normally distributed random variable with mean $0$ and variance $1$ (any random variable with support equal to $\reals$ and absolutely continuous distribution will do). Also, let $\sigma_n$ be as in the previous example, defined by equation (\ref{eqn:sn defn for approx poisson}). Consider the SDE
\begin{align*}
&dY^n_t=f(nZ^n_t)\sigma_n(Y^n_t)\,dW_t,\\
&dZ^n_t=0
\end{align*}
where $Y^n_0=0$ and $Z^n_0=U$, and $W$ is a standard Brownian motion.
As $f(nx)\sigma_n(x)$ is Lipschitz continuous, the processes $(Y^n,Z^n)$ will be strong Markov martingales.

It is easy to solve this SDE. Let $X^n$ be the processes defined in the previous example. Then a solution is given by
\begin{align*}
&Y^n_t=f(nU)X^n_t,\\
&Z^n_t=U.
\end{align*}

From the previous example, we know that $X^n$ converges to a symmetric Poisson process $X$ in the sense of finite-dimensional distributions.
Also, from the definition of $f$, $f(nU)$ will converge weakly to the uniform distribution on $[0,1]$. So, let $V$ be a random variable uniformly distributed on $[0,1]$ and suppose that $X$, $V$ and $U$ are independent. Setting
\begin{align*}
&Y_t=V X_t,\\
&Z_t=U
\end{align*}
then $(Y^n,Z^n)\rightarrow(Y,Z)$ in the sense of finite-dimensional distributions. This process is both discontinuous and non-Markov, showing that the results of this paper do not extend to two dimensional processes.

In fact, I conjecture that for $d>1$ \emph{any} $d$-dimensional cadlag stochastic process is a limit of martingale diffusions in the sense of finite-dimensional distributions, and for $d>2$ any such process is a limit of homogeneous martingale diffusions.

\section{Almost-Continuity}
\label{sec:ac}

We split the proof of Theorem \ref{thm:limit of acd is acd} into two main parts. First, in this section, we show that the limit is an almost-continuous process, and we leave the proof that it is strong Markov until later.
The main result that we shall prove in this section is the following.
\begin{lemma}\label{lemma:lim of acd is ac}
Let $(\PP_n)_{n\in\nat}$ be a sequence of probability measures on $(\D,\setsF)$ under which $X$ is an almost-continuous diffusion.
If $\PP_n\rightarrow\PP$ in the sense of finite-dimensional distributions on a dense subset of $\reals_+$ and $X$ is continuous in probability under $\PP$, then it is almost-continuous under $\PP$.
\end{lemma}

The method we use will be to reformulate the pathwise `almost-continuity' property into a condition on the finite distributions of $X$. The idea is that given real numbers (or more generally, subsets of the reals) $x<y$ and $x^\prime<y^\prime$ then a coupling argument can be used to show that the probability of $X$ going from $x$ to $y^\prime$ multiplied by the probability of going from $y$ to $x^\prime$ across a time interval $[s,t]$ is bounded by the probability of going from $x$ to $x^\prime$ multiplied by the probability of going from $y$ to $y^\prime$. The precise statement is as follows.

\begin{lemma}\label{lemma:finite dist of ac process}
Let $\PP$ be a probability measure on $(\D,\setsF)$ under which $X$ is continuous in probability. Then each of the following statements implies the next.
\begin{enumerate}
\item\label{item:lemma:acd} $X$ is an almost-continuous diffusion.
\item\label{item:lemma:ac cond} for every $s<t\in\reals_+$, non-negative $\setsF_s$-measurable random variables $U,V$, and real numbers $a$ and $b<c\le d<e$, then
\begin{equation}\label{eqn:lemma:finite dist of ac process:1}\begin{split}
&\E{U1_{\left\{X_s<a,\ d<X_t<e\right\}}}
\E{V1_{\left\{X_s>a,\ b<X_t<c\right\}}}\\
&\le
\E{U1_{\left\{X_s<a,\ b<X_t<c\right\}}}
\E{V1_{\left\{X_s>a,\ d<X_t<e\right\}}}.
\end{split}\end{equation}
\item $X$ is almost-continuous.
\end{enumerate}
\end{lemma}

We shall split the proof of this lemma into several parts. The approach that we use is to consider two independent copies of $X$ and look at the first time that they cross. So, we start by defining the probability space on which these processes exist, which is just the product of $(\D,\setsF)$ with itself.
\begin{equation}\label{eqn:D2 F2}\begin{split}
&\D^2 = \D\times\D,\\
&\setsF^2 = \setsF\otimes\setsF.
\end{split}\end{equation}
Then we let $Y$ and $Z$ be the coordinate processes,
\begin{equation}\label{eqn:Y Z}\begin{split}
&Y,Z:\reals_+\times\D^2\rightarrow\reals,\\
&Y_t(\omega_1,\omega_2)\equiv X_t(\omega_1)=\omega_1(t),\\
&Z_t(\omega_1,\omega_2)\equiv X_t(\omega_2)=\omega_2(t).
\end{split}\end{equation}
We also write $\setsF^2_t$ for the filtration generated by $Y$ and $Z$, which is just the product of $\setsF_t$ with itself.
\begin{equation}\label{eqn:F2t}
\setsF^2_t\equiv\setsF_t\otimes\setsF_t=\sigma\left(Y_s,Z_s:s\in[0,t]\right).
\end{equation}
Given any probability measure $\PP$ on $(\D,\setsF)$ we denote the measure on $(\D^2,\setsF^2)$ formed by the product of $\PP$ with itself by $\tilde\PP$.
\begin{equation}\label{eqn:Ptilde}
\tilde\PP\equiv\PP\otimes\PP.
\end{equation}
In what follows, the notation $\Et{\cdot}$ will be used to denote expectations with respect to the measure $\tilde{\PP}$.
From these definitions, $Y$ and $Z$ are adapted cadlag processes, and under $\tilde\PP$ they are independent and identically distributed each with the same distribution as $X$ has under $\PP$. We now rewrite statement \ref{item:lemma:ac cond} of Lemma \ref{lemma:finite dist of ac process} in terms of the finite distributions of $Y$ and $Z$.

\begin{lemma}\label{lemma:finite dist ac 2 dbl}
Given any probability measure $\PP$ on $(\D,\setsF)$,
statement \ref{item:lemma:ac cond} of Lemma \ref{lemma:finite dist of ac process} is equivalent to the statement that for every $s<t\in\reals_+$ and real numbers $b<c\le d<e$ then,
\begin{equation}\label{lemma:finite dist ac 2 dbl:1}\begin{split}
&\Probt{Y_s<Z_s,\,b<Z_t<c,\,d<Y_t<e|\setsF^2_s}\\
&\le\Probt{Y_s<Z_s,\,b<Y_t<c,\,d<Z_t<e|\setsF^2_s}.
\end{split}\end{equation}
\end{lemma}
\begin{proof}
First, suppose that inequality (\ref{lemma:finite dist ac 2 dbl:1}) holds. Choose $s<t\in\reals_+$ and real numbers $a$ and $b<c\le d<e$. Also choose non-negative $\setsF_s$-measurable random variables $U=u(X)$ and $V=v(X)$. Then the definition (\ref{eqn:Ptilde}) of $\tilde\PP$ together with inequality (\ref{lemma:finite dist ac 2 dbl:1}) gives
\begin{equation*}\begin{split}
&\E{U1_{\left\{X_s<a,\,d<X_t<e\right\}}}
\E{V1_{\left\{X_s>a,\,b<X_t<c\right\}}}\\
={}&\Et{u(Y)v(Z)1_{\left\{Y_s<a<Z_s\right\}}1_{\left\{d<Y_t<e\right\}}1_{\left\{b<Z_t<c\right\}}}\\
={}&\Et{u(Y)v(Z)1_{\left\{Y_s<a<Z_s\right\}}\Probt{Y_s<Z_s,\,b<Z_t<c,\,d<Y_t<e|\setsF^2_s}}\\
\le{}&\Et{u(Y)v(Z)1_{\left\{Y_s<a<Z_s\right\}}\Probt{Y_s<Z_s,\,b<Y_t<c,\,d<Z_t<e|\setsF^2_s}}\\
={}&\Et{u(Y)v(Z)1_{\left\{Y_s<a<Z_s\right\}}1_{\left\{b<Y_t<c\right\}}1_{\left\{d<Z_t<e\right\}}}\\
={}&\E{U1_{\left\{X_s<a,\,b<X_t<c\right\}}}
\E{V1_{\left\{X_s>a,\,d<X_t<e\right\}}}
\end{split}\end{equation*}
as required.

Conversely, suppose that statement \ref{item:lemma:ac cond} of Lemma \ref{lemma:finite dist of ac process} holds. Now choose $s<t\in\reals_+$, real numbers $a^\prime<a$ and $b<c\le d<e$ and bounded non-negative $\setsF_s$-measurable random variables $U=u(X)$ and $V=v(X)$.
Defining the $\setsF^2_s$-measurable random variable $W=u(Y)v(Z)$, then the definition (\ref{eqn:Ptilde}) of $\tilde\PP$ together with inequality (\ref{eqn:lemma:finite dist of ac process:1}) gives
\begin{equation*}\begin{split}
&\Et{W1_{\left\{a^\prime\le Y_s<a<Z_s,\,b<Z_t<c,\,d<Y_t<e\right\}}}\\
={}&\E{\left(u(X)1_{\left\{a^\prime\le X_s\right\}}\right)1_{\left\{X_s<a,\,d<X_t<e\right\}}}
\E{v(X)1_{\left\{a<X_s\right\}}1_{\left\{b<X_t<c\right\}}}\\
\le{}&\E{\left(u(X)1_{\left\{a^\prime\le X_s\right\}}\right)1_{\left\{X_s<a,\,b<X_t<c\right\}}}
\E{v(X)1_{\left\{a<X_s\right\}}1_{\left\{d<X_t<e\right\}}}\\
={}&\Et{W1_{\left\{a^\prime\le Y_s<a<Z_s,\,b<Y_t<c,\,d<Z_t<e\right\}}}.
\end{split}\end{equation*}
For any $\epsilon>0$ we can set $a^\prime=(n-1)\epsilon$ and $a=n\epsilon$ in this inequality and sum over $n$,
\begin{equation*}\begin{split}
&\Et{W1_{\left\{\exists n\in\ints\textrm{ s.t.\,}Y_s<n\epsilon<Z_s,\,b<Z_t<c,\,d<Y_t<e\right\}}}\\
={}&\sum_{n=-\infty}^\infty\Et{W1_{\left\{(n-1)\epsilon\le Y_s<n\epsilon<Z_s,\,b<Z_t<c,\,d<Y_t<e\right\}}}\\
\le{}&\sum_{n=-\infty}^\infty\Et{W1_{\left\{(n-1)\epsilon\le Y_s<n\epsilon<Z_s,\,b<Y_t<c,\,d<Z_t<e\right\}}}\\
={}&\Et{W1_{\left\{\exists n\in\ints\textrm{ s.t.\,}Y_s<n\epsilon<Z_s,\,b<Y_t<c,\,d<Z_t<e\right\}}}
\end{split}\end{equation*}
Letting $\epsilon$ decrease to $0$ and using bounded convergence gives
\begin{equation}\label{eqn:proof:finite dist ac 2 dbl:2}
\Et{W1_{\left\{Y_s<Z_s,\,b<Z_t<c,\,d<Y_t<e\right\}}}
\le \Et{W1_{\left\{Y_s<Z_s,\,b<Y_t<c,\,d<Z_t<e\right\}}}.
\end{equation}
Finally, we note that the set of bounded and non-negative $\setsF^2_s$-measurable random variables $W$ for which inequality (\ref{eqn:proof:finite dist ac 2 dbl:2}) holds is closed under taking positive linear combinations, and under taking increasing and decreasing limits. Therefore, inequality (\ref{eqn:proof:finite dist ac 2 dbl:2}) holds for all bounded and non-negative $\setsF^2_s$-measurable random variables $W$, and inequality (\ref{lemma:finite dist ac 2 dbl:1}) follows from this.
\end{proof}
Using this result, it is now easy to prove that the first statement of Lemma \ref{lemma:finite dist of ac process} implies the second. The idea is to look at the processes $Y$ and $Z$ up until the first time that they touch, which is similar to the coupling method used in \citep{Hobson} to investigate the conditional expectations of convex functions of a martingale diffusion.
\begin{lemma}\label{lemma:acd gives finite dist of ac process}
If $\PP$ is a probability measure on $(\D,\setsF)$ under which $X$ is an almost-continuous diffusion then statement \ref{item:lemma:ac cond} of Lemma \ref{lemma:finite dist of ac process} holds.
\end{lemma}
\begin{proof}
First choose real numbers $b<c\le d<e$, times $s<t\in\reals_+$, and set
\begin{equation*}
g_1(x)=1_{\{b<x<c\}},\
g_2(x)=1_{\{d<x<e\}}.
\end{equation*}
Then, by the strong Markov property, there exist measurable functions $f_1,f_2:[0,t]\times\reals\rightarrow\reals$ such that
\begin{equation*}
1_{\{\tau\le t\}}f_i(\tau,X_\tau)=1_{\{\tau\le t\}}\E{g_i(X_t)|\setsF_\tau}
\end{equation*}
for $i=1,2$ and for every stopping time $\tau$.
This follows easily from definition \ref{defn:acd} of the strong Markov property (see \citep{Lowther1}, Lemma 2.1). Furthermore it then follows that
\begin{align*}
&1_{\{\tau\le t\}}f_i(\tau,Y_\tau)=1_{\{\tau\le t\}}\Et{g_i(Y_t)|\setsF^2_\tau}\\
&1_{\{\tau\le t\}}f_i(\tau,Z_\tau)=1_{\{\tau\le t\}}\Et{g_i(Z_t)|\setsF^2_\tau}
\end{align*}
for every $\setsF^2_\cdot$-stopping time $\tau$ (see \citep{Lowther1}, Lemma 2.2).

Now let $\tau$ be the following stopping time.
\begin{equation*}
\tau=\left\{
\begin{array}{ll}
\inf\left\{u\in[s,\infty):Y_u\ge Z_u\right\},&\textrm{if }Y_s<Z_s,\\
\infty,&\textrm{otherwise}.
\end{array}
\right.
\end{equation*}
Strictly speaking, $\tau$ will only be a stopping time with respect to the universal completion of the filtration. So, throughout this section we assume that all \salg s are replaced by their universal completions. Note that if $Y_s<Z_s$ and $\tau>t$ then $Y_t<Z_t$ so $g_1(Z_t)g_2(Y_t)=0$. Therefore
\begin{equation*}\begin{split}
\tilde\PP\left(Y_s<Z_s,\,b<Z_t<c,\,d<Y_t<e|\setsF^2_s\right)
&=\Et{1_{\{\tau\le t\}}g_1(Z_t)g_2(Y_t)|\setsF^2_s}\\
&=\Et{1_{\{\tau\le t\}}f_1(\tau,Z_\tau)f_2(\tau,Y_\tau)|\setsF^2_s}.
\end{split}\end{equation*}
However, by almost-continuity, we have $Y_\tau=Z_\tau$ whenever $\tau<\infty$ ($\tilde\PP$ a.s.). So,
\begin{equation*}\begin{split}
&\tilde\PP\left(Y_s<Z_s,\,b<Z_t<c,\,d<Y_t<e|\setsF^2_s\right)\\
={}&\Et{1_{\{\tau\le t\}}f_1(\tau,Y_\tau)f_2(\tau,Z_\tau)|\setsF^2_s}\\
={}&\Et{1_{\{\tau\le t\}}g_1(Y_t)g_2(Z_t)|\setsF^2_s}\\
={}&\tilde\PP\left(\tau\le t,\,b<Y_t<c,\,d<Z_t<e|\setsF^2_s\right)\\
\le{}&\tilde\PP\left(Y_s<Z_s,\,b<Y_t<c,\,d<Z_t<e|\setsF^2_s\right)
\end{split}\end{equation*}
The result now follows from Lemma \ref{lemma:finite dist ac 2 dbl}.
\end{proof}

To prove that the second statement of Lemma \ref{lemma:finite dist of ac process} implies the third, we shall look at what happens when the processes $Y$ and $Z$ first cross after any given time. The idea is to show that they cannot jump past each other at this time, and therefore will be equal. As this will be a stopping time we start by rewriting statement \ref{item:lemma:ac cond} of Lemma \ref{lemma:finite dist of ac process} in terms of the distribution at a stopping time.

\begin{lemma}\label{lemma:finite dist of ac 2 pathwise}
Let $\PP$ be a probability measure on $(\D,\setsF)$
such that statement \ref{item:lemma:ac cond} of Lemma \ref{lemma:finite dist of ac process} holds.

Let $b<c\le d<e$ be real numbers and set $V=(b,c)\times(d,e)$.
Also let $U$ be an open subset of $\{(x,y)\in\reals^2:x<y\}$ that is disjoint from $V$, and for any $\setsF^2_\cdot$-stopping time $S$ define the stopping time
\begin{equation}\label{eqn:lemma:finite dist of ac 2 pathwise:1}
\tau^U_{S}=\left\{
\begin{array}{ll}
\inf\left\{t>S:(Y_t,Z_t)\not\in U\right\},&\textrm{if $S<\infty$ and $(Y_S,Z_S)\in U$},\\
\infty,&\textrm{otherwise}.
\end{array}
\right.
\end{equation}
Then,
\begin{equation*}
\tilde\PP\left(\tau^U_S<\infty,\,(Z_{\tau^U_S},Y_{\tau^U_S})\in V\right)
\le\tilde\PP\left(\tau^U_S<\infty,\,(Y_{\tau^U_S},Z_{\tau^U_S})\in V\right).
\end{equation*}
\end{lemma}
\begin{proof}
Let $t_{n,k}=k/n$ for all $k\in\ints_{\ge 0}$ and $n\in\nat$, and set
\begin{equation*}
A_{n,k} = \left\{S\le t_{n,k}<\tau^U_S,\,(Y_S,Z_S)\in U\right\}\in\setsF^2_{t_{n,k}}.
\end{equation*}
We now let $T_n$ be the stopping time
\begin{equation*}
T_n=\inf\left\{t_{n,k}:k\in\nat,\,t_{n,k}\ge\tau^U_S>t_{n,k-1}\ge S\right\}
\end{equation*}
so that $T_n\downarrow\tau^U_S$ as $n\rightarrow\infty$.
Then we can apply Lemma \ref{lemma:finite dist ac 2 dbl},
\begin{equation}\label{eqn:proof:finite dist of ac 2 pathwise:1}\begin{split}
\tilde\PP\left(T_n<\infty,(Z_{T_n},Y_{T_n})\in V\right)
&=\sum_{k=1}^\infty\tilde\PP\left(T_n=t_{n,k},(Z_{t_{n,k}},Y_{t_{n,k}})\in V\right)\\
&=\sum_{k=1}^\infty\tilde\PP\left(A_{n,k-1}\cap\left\{(Z_{t_{n,k}},Y_{t_{n,k}})\in V\right\}\right)\\
&\le\sum_{k=1}^\infty\tilde\PP\left(A_{n,k-1}\cap\left\{(Y_{t_{n,k}},Z_{t_{n,k}})\in V\right\}\right)\\
&=\sum_{k=1}^\infty\tilde\PP\left(T_n=t_{n,k},\,(Y_{t_{n,k}},Z_{t_{n,k}})\in V\right)\\
&=\tilde\PP\left(T_n<\infty,\,(Y_{T_n},Z_{T_n})\in V\right)
\end{split}\end{equation}
Now suppose that $\tau^U_S<\infty$ and $(Z_{\tau^U_S},Y_{\tau^U_S})\in V$. As $T_n\downarrow\tau^U_S$ as $n\rightarrow\infty$, the right-continuity of $Y$ and $Z$ gives $(Z_{T_n},Y_{T_n})\in V$ for large $n$. So, by bounded convergence
\begin{equation}\label{eqn:proof:finite dist of ac 2 pathwise:2}
\tilde\PP\left(\tau^U_S<\infty,\,(Z_{\tau^U_S},Y_{\tau^U_S})\in V\right)
\le\liminf_{n\rightarrow\infty}\tilde\PP\left(T_n<\infty,\,(Z_{T_n},Y_{T_n})\in V\right).
\end{equation}
Similarly, suppose that $(Y_{T_n},Z_{T_n})\in V$ for infinitely many $n$. By the right-continuity of $Y$ and $Z$, this gives $b\le Y_{\tau^U_S}\le c$ and $d\le Z_{\tau^U_S}\le e$. So,
\begin{equation*}
\limsup_{n\rightarrow\infty}1_{\left\{T_n<\infty,\,(Y_{T_n},Z_{T_n})\in V\right\}}
\le1_{\left\{\tau^U_S<\infty,\,b\le Y_{\tau^U_S}\le c,\,d\le Z_{\tau^U_S}\le e\right\}}.
\end{equation*}
Then monotone convergence gives
\begin{equation}\label{eqn:proof:finite dist of ac 2 pathwise:3}\begin{split}
&\limsup_{n\rightarrow\infty}\tilde\PP\left(T_n<\infty,\,(Y_{T_n},Z_{T_n})\in V\right)\\
&\le
\tilde\PP\left(\tau^U_S<\infty,\,b\le Y_{\tau^U_S}\le c,\,d\le Z_{\tau^U_S}\le e\right).
\end{split}\end{equation}
Combining inequalities (\ref{eqn:proof:finite dist of ac 2 pathwise:1}), (\ref{eqn:proof:finite dist of ac 2 pathwise:2}) and (\ref{eqn:proof:finite dist of ac 2 pathwise:3}) gives
\begin{equation}\label{eqn:proof:finite dist of ac 2 pathwise:4}\begin{split}
&\tilde\PP\left(\tau^U_S<\infty,\,(Z_{\tau^U_S},\,Y_{\tau^U_S})\in V\right)\\
&\le
\tilde\PP\left(\tau^U_S<\infty,\,b\le Y_{\tau^U_S}\le c,\,d\le Z_{\tau^U_S}\le e\right).
\end{split}\end{equation}
Finally, set $b_n=b+1/n$, $c_n=c-1/n$, $d_n=d+1/n$ and $e_n=e-1/n$ for every $n\in\nat$. Then inequality (\ref{eqn:proof:finite dist of ac 2 pathwise:4}) with $(b_n,c_n)\times(d_n,e_n)$ in place of $V$ gives
\begin{equation*}\begin{split}
&\tilde\PP\left(\tau^U_S<\infty,\,(Z_{\tau^U_S},Y_{\tau^U_S})\in V\right)\\
={}&\lim_{n\rightarrow\infty}\PP\left(\tau^U_S<\infty,\,b_n<Z_{\tau^U_S}<c_n,\,d_n<Y_{\tau^U_S}<e_n\right)\\
\le{}&\limsup_{n\rightarrow\infty}\PP\left(\tau^U_S<\infty,\,b_n\le Y_{\tau^U_S}\le c_n,\,d_n\le Z_{\tau^U_S}\le e_n\right)\\
={}&
\tilde\PP\left(\tau^U_S<\infty,\,(Y_{\tau^U_S},Z_{\tau^U_S})\in V\right).\qedhere
\end{split}\end{equation*}
\end{proof}

We shall use Lemma \ref{lemma:finite dist of ac 2 pathwise} to prove almost-continuity by showing that the probability of $Y$ jumping from strictly below $Z$ to above it is bounded by the probability of them jumping simultaneously. The following simple result will tell us that $Y$ and $Z$ cannot jump simultaneously.

\begin{lemma}\label{lemma:no simultaneous jumps}
Let $Y$ and $Z$ be independent cadlag processes such that $Y$ is continuous in probability.
Then, with probability $1$, $Y_{t-}=Y_t$ or $Z_{t-}=Z_t$ for every $t>0$.
\end{lemma}
\begin{proof}
As $Y$ is cadlag, there exist $Y$-measurable random times $(S_n)_{n\in\nat}$ such that $\cup_{n\in\nat}\slbrack S_n\srbrack$ contains all the jump times of $Y$ almost-surely (see \citep{HeWangYan} Theorem 3.32). Without loss of generality, we may suppose that $Y_{S_n-}\not=Y_{S_n}$ whenever $S_n<\infty$. Similarly, there exist $Z$-measurable random times $(T_n)_{n\in\nat}$ such that $\cup_{n\in\nat}\slbrack T_n\srbrack$ contains all the jump times of $Z$ almost-surely, and such that $Z_{T_n-}\not=Z_{T_n}$ whenever $T_n<\infty$.
Then,
\begin{equation*}
\Prob{\exists t\in\reals_+\textrm{ s.t.\,}Y_{t-}\not=Y_t\textrm{ and }Z_{t-}\not=Z_t}\le\sum_{m,n=1}^\infty \Prob{S_m=T_n<\infty}.
\end{equation*}
However, the independence of $S_m$ and $T_n$ together with the continuity in probability of $Y$ gives
\begin{equation*}\begin{split}
\Prob{S_m=T_n<\infty}&=\sum_{t\in\reals_+}\Prob{S_m=t}\Prob{T_n=t}\\
&\le\sum_{t\in\reals_+}\Prob{Y_{t-}\not=Y_t}\Prob{Z_{t-}\not=Z_t}\\
&=0.\qedhere
\end{split}\end{equation*}
\end{proof}

This simple result together with Lemma \ref{lemma:finite dist of ac 2 pathwise} gets us some way towards showing that $X$ is almost-continuous.

\begin{lemma}\label{lemma:no jumps across Z}
Let $\PP$ be a probability measure on $(\D,\setsF)$ under which $X$ is continuous in probability, and such that statement \ref{item:lemma:ac cond} of Lemma \ref{lemma:finite dist of ac process} holds. Then
\begin{equation*}
\tilde\PP\left(\exists t>0\rm{\ s.t.\,}Y_{t-}<Z_t<Y_t\right)=0.
\end{equation*}
\end{lemma}
\begin{proof}
Choose any real numbers $b<c<d$ and let $U,V$ be the sets
\begin{align*}
&U=(-\infty,b)\times(b,c),\\
&V=(b,c)\times(c,d).
\end{align*}
Then $U\cap V=\emptyset$, so letting $\tau^U_s$ be the stopping time given by equation (\ref{eqn:lemma:finite dist of ac 2 pathwise:1}) for any $s\in\reals_+$, we can apply Lemma \ref{lemma:finite dist of ac 2 pathwise} to get
\begin{equation}\label{eqn:proof:no jumps across Z:1}
\tilde\PP\left(\tau^U_s<\infty,\,(Z_{\tau^U_s},Y_{\tau^U_s})\in V\right)
\le\tilde\PP\left(\tau^U_s<\infty,\,(Y_{\tau^U_s},Z_{\tau^U_s})\in V\right).
\end{equation}
However, if $\tau^U_s<\infty$ and $(Y_{\tau^U_s},Z_{\tau^U_s})\in V$ then $Y_{\tau^U_s}>b\ge Y_{\tau^U_s-}$ and $Z_{\tau^U_s}>c\ge Z_{\tau^U_s-}$.
By Lemma \ref{lemma:no simultaneous jumps} the processes $Y$ and $Z$ cannot jump simultaneously, so this has zero probability. Inequality (\ref{eqn:proof:no jumps across Z:1}) then gives
\begin{equation}\label{eqn:proof:no jumps across Z:2}
\tilde\PP\left(\tau^U_s<\infty,\,(Z_{\tau^U_s},Y_{\tau^U_s})\in V\right)=0.
\end{equation}
Now suppose that $(Y_{t-},Z_{t-})\in U$ and $(Z_t,Y_t)\in V$ for some time $t$. Then, by left-continuity, there exists an $s<t$ such that $s\in\rats_+$ and $(Y_u,Z_u)\in U$ for every $u\in[s,t)$. In this case $\tau^U_s=t$.
Therefore equation (\ref{eqn:proof:no jumps across Z:2}) gives
\begin{equation}\label{eqn:proof:no jumps across Z:3}\begin{split}
&\tilde\PP\left(\exists t\in\reals_+\textrm{ s.t.\,} Y_{t-}<b<Z_{t-}=Z_t<c<Y_t<d\right)\\
\le{}&\tilde\PP\left(\exists t\in\reals_+\textrm{ s.t.\,} (Y_{t-},Z_{t-})\in U,\,(Z_t,Y_t)\in V\right)\\
\le{}&\sum_{s\in\rats_+}\tilde\PP\left(\tau^U_s<\infty,\,(Z_{\tau^U_s},Y_{\tau^U_s})\in V\right)\\
={}&0.
\end{split}\end{equation}
Note that for every $t$ such that $Y_{t-}<Z_t<Y_t$ then Lemma \ref{lemma:no simultaneous jumps} tells us that $Z_{t-}=Z_t$. So, by (\ref{eqn:proof:no jumps across Z:3})
\begin{equation*}\begin{split}
&\tilde\PP\left(\exists t\in\reals_+\textrm{ s.t.\,}Y_{t-}<Z_t<Y_t\right)\\
={}&
\tilde\PP\left(\exists t\in\reals_+\textrm{ s.t.\,} Y_{t-}<Z_{t-}=Z_t<Y_t\right)\\
\le{}&\sum_{a<b<c<d\in\rats}\tilde\PP\left(\exists t\in\reals_+\textrm{ s.t.\,} Y_{t-}<b<Z_{t-}=Z_t<c<Y_t<d\right)\\
={}&0.\qedhere
\end{split}\end{equation*}
\end{proof}

Lemma \ref{lemma:no jumps across Z} shows that $Y$ cannot jump from strictly below to strictly above Z. However, it does not rule out the possibility that $Y$ can approach $Z$ from below, then jump to above $Z$ (which would contradict almost-continuity). In order to show that this behaviour is not possible, we shall again make use of Lemma \ref{lemma:finite dist of ac 2 pathwise}. The idea is to reduce it to showing that it is not possible for $Y$ to approach $Z$ from below, then jump downwards. In fact this behaviour is ruled out by the conclusion of Lemma \ref{lemma:no jumps across Z}, but it is far from obvious that this is the case. We shall make use of some results that we proved in \citep{Lowther1}. First, we restate the definition of the marginal support used in \citep{Lowther1}.

\begin{definition}\label{defn:marginal support}
Let $X$ be a real valued stochastic process. Then, its \emph{marginal support} is
\begin{displaymath}
\msupport{X} = \left\{(t,x)\in\halfplane: x \in\support{X_t}\right\}.
\end{displaymath}
\end{definition}
As we showed in \citep{Lowther1}, the marginal support of a process $X$ is Borel measurable, and the relevance of the marginal support to our current argument is given by the following result.
\begin{lemma}\label{lemma:never crosses MSupp cond}
If $X$ is a cadlag real valued process which is continuous in probability then the following are equivalent.
\begin{enumerate}
\item The set
\begin{equation*}
\left\{(t,x)\in\halfplane:X_{t-}<x<X_t\right\}.
\end{equation*}
is disjoint from $\msupport{X}$ with probability one.
\item Given two independent cadlag processes $Y$ and $Z$, each with the same distribution as $X$, then
\begin{equation*}
\Prob{\exists t>0{\rm\ s.t.\,} Y_{t-}<Z_t<Y_t}=0
\end{equation*}
\end{enumerate}
\end{lemma}
\begin{proof}See \citep{Lowther1}, Lemma 4.7.
\end{proof}
We also make use of the following result, which says that it is not possible for $Y$ to approach $Z$ from below and then jump downwards to a value strictly less than $Z$.
\begin{lemma}\label{lemma:mgale never crosses msupp then ac:increasing times}
Let $X$ be a cadlag real valued process which is continuous in probability, and such that the set
\begin{equation*}
\left\{(t,x)\in\halfplane:X_{t-}<x<X_t\textrm{ or }X_t<x<X_{t-}\right\}
\end{equation*}
is disjoint from $\msupport{X}$ with probability one.

Also, let $Y$ and $Z$ be independent cadlag processes each with the same distribution as $X$.
For any $s\in\reals_+$ let $T$ be the random time
\begin{equation*}
T = \left\{
\begin{array}{ll}
\inf\{t\in\reals_+:t\ge s,\,Y_t\ge Z_t\},& \textrm{if }Y_s<Z_s,\\
\infty,&\textrm{otherwise},
\end{array}
\right.
\end{equation*}
and $(T_n)_{n\in\nat}$ be the random times
\begin{equation*}
T_n = \left\{
\begin{array}{ll}
\inf\{t\in\reals_+:t\ge s,\,Y_t+1/n\ge Z_t\},& \textrm{if }Y_s<Z_s,\\
\infty,&\textrm{otherwise}.
\end{array}
\right.
\end{equation*}
Then $T_n\uparrow T$ as $n\rightarrow\infty$ (a.s.). Also, $T_n<T$ whenever $T<\infty$ and $Y_T\not=Z_T$ (a.s.).
\end{lemma}
\begin{proof}
See \citep{Lowther1}, Lemma 4.9.
\end{proof}

We can now combine these results to prove Lemma \ref{lemma:finite dist of ac process}. As we mentioned previously, the idea is to show that it is not possible for $Y$ to approach $Z$ from below and then jump past it.
\begin{proof}
First, statement \ref{item:lemma:acd} implies statement \ref{item:lemma:ac cond} by Lemma \ref{lemma:acd gives finite dist of ac process}. So, we now suppose that statement \ref{item:lemma:ac cond} holds. Then by Lemma \ref{lemma:no jumps across Z} we have
\begin{equation*}
\tilde\PP\left(\exists t>0\rm{\ s.t.\,}Y_{t-}<Z_t<Y_t\right)=0.
\end{equation*}
Applying Lemma \ref{lemma:never crosses MSupp cond} shows that the set
\begin{equation*}
\left\{(t,x)\in\halfplane:X_{t-}<x<X_t\right\}
\end{equation*}
is disjoint from $\msupport{X}$ with probability one. Similarly, we can apply the same argument to $-X$ to see that
\begin{equation*}
\left\{(t,x)\in\halfplane:X_{t}<x<X_{t-}\right\}
\end{equation*}
is also disjoint from $\msupport{X}$ with probability one. Therefore, the requirements of Lemma \ref{lemma:mgale never crosses msupp then ac:increasing times} are satisfied. For any $s\in\reals_+$ and $n\in\nat$ let $T$ and $T_n$ be the stopping times defined by Lemma \ref{lemma:mgale never crosses msupp then ac:increasing times}. Also, define the stopping times
\begin{align*}
&S_n=\left\{
\begin{array}{ll}
T_n,&\textrm{if $T_n<T$},\\
\infty,&\textrm{otherwise},
\end{array}
\right.\\
&S=\left\{
\begin{array}{ll}
T,&\textrm{if $T_n<T$ for every $n\in\nat$},\\
\infty,&\textrm{otherwise}.
\end{array}
\right.
\end{align*}
By Lemma \ref{lemma:mgale never crosses msupp then ac:increasing times}, $T_n\uparrow T$, and so $S_n\uparrow\uparrow S$ whenever $S<\infty$.
Now let $A$ be the set
\begin{equation*}
A = \left\{a\in\reals:\tilde\PP\left(S<\infty,\,Z_{S-}=a\right)=0\right\}.
\end{equation*}
As $\reals\setminus A$ is countable, we see that $A$ is a dense subset of $\reals$.
We now choose any $b<c<d\in A$ and set
\begin{align*}
&U = \left\{(x,y)\in\reals:x<y<c\right\},\\
&V = (b,c)\times(c,d).
\end{align*}
Now fix any $t>s$ and let $S^\prime_n$ be the stopping time,
\begin{equation*}
S^\prime_n=\left\{
\begin{array}{ll}
S_n,&\textrm{if }S_n\le t,\\
\infty,&\textrm{otherwise}.
\end{array}
\right.
\end{equation*}
Also let $\tau^U_{S^\prime_n}$ be the stopping time defined by equation (\ref{eqn:lemma:finite dist of ac 2 pathwise:1}). Then, by Lemma \ref{lemma:finite dist of ac 2 pathwise}
\begin{equation}\label{eqn:proof:finite dist of ac process:1}
\tilde\PP\left(\tau^U_{S^\prime_n}<\infty,\,(Z_{\tau^U_{S^\prime_n}},Y_{\tau^U_{S^\prime_n}})\in V\right)
\le\tilde\PP\left(\tau^U_{S^\prime_n}<\infty,\,(Y_{\tau^U_{S^\prime_n}},Z_{\tau^U_{S^\prime_n}})\in V\right).
\end{equation}
Also, if $S<\infty$ then $S=T$ so, by the definition of $T$ we have $Y_S\ge Z_S$. So, $(Y_{S},Z_{S})\not\in U$. Now consider the following cases,
\begin{itemize}
\item $S\le t$ and $Z_{S-}<c$. Then, as $S_n\uparrow\uparrow S$, we see that $\tau^U_{S^\prime_n}=S$ for large $n$.
\item $S\le t$ and $Z_{S-}> c$. Then $Z_{S_n}>c$ for large $n$ and so $\tau^U_{S^\prime_n}=\infty$.
\item $S> t$. Then $S^\prime_n=\infty$ for large $n$.
\end{itemize}
The case where $Z_{S-}=c$ is ruled out because we chose $c\in A$. 
Therefore, we can take the limit as $n$ goes to infinity in inequality (\ref{eqn:proof:finite dist of ac process:1}),
\begin{align*}
&\tilde\PP\left(S\le t,\,Z_{S-}<c,\,(Z_S,Y_S)\in V\right)\\
&\le\tilde\PP\left(S\le t,\,Z_{S-}<c,\,(Y_S,Z_S)\in V\right)
\end{align*}
As $Y_S\ge Z_S$, the right hand side of this inequality is $0$,
\begin{equation*}
\tilde\PP\left(S\le t,\,Z_{S-}<c,\,b<Z_S<c<Y_S<d\right)=0.
\end{equation*}
Therefore, letting $B$ be any countable and dense subset of $A$,
\begin{align*}
&\tilde\PP\left(S\le t,\,Z_{S-}=Z_S<Y_S\right)\\
\le{}&\sum_{b<c<d\in B}\PP\left(S\le t,\,b<Z_{S-}=Z_S<c<Y_S<d\right)\\
={}&0.
\end{align*}
This shows that it is not possible for $Y$ to approach $Z$ from below, then jump upwards to above $Z$.
Similarly, replacing $(Y,Z)$ by $(-Z,-Y)$ in the above argument gives
\begin{equation*}
\tilde\PP\left(S\le t,\,Z_S<Y_S=Y_{S-}\right)=0.
\end{equation*}
Lemma \ref{lemma:no simultaneous jumps} says that $Z_{S-}=Z_S$ or $Y_{S-}=Y_S$ whenever $S<\infty$,
\begin{equation*}\begin{split}
\tilde\PP\left(S\le t\,,\,Z_S<Y_S\right)={}&\tilde\PP\left(S\le t,\,Z_{S-}=Z_S<Y_S\right)\\
&+\tilde\PP\left(S\le t,\,Z_S<Y_S=Y_{S-}\right)\\
={}&0.
\end{split}\end{equation*}
That is, $Y_S=Z_S$ whenever $S\le t$ (a.s.).
Finally, from the statement of Lemma \ref{lemma:mgale never crosses msupp then ac:increasing times}, we know that $Y_T=Z_T$ whenever $T_n=T<\infty$. So, $Y_T=Z_T$ whenever $T\le t$.
So whenever $Y_s<Z_s$ and $Z_t<Y_t$ we have $s<T<t$ and $Z_T=Y_T$. Therefore
\begin{equation*}
\tilde\PP\left(Y_s<Z_s,\,Y_t>Z_t{\rm\ and\ }Y_u\not= Z_u\textrm{ for every }u\in(s,t)\right)=0.\qedhere
\end{equation*}
\end{proof}

We now move on to the proof of Lemma \ref{lemma:lim of acd is ac}. We start off by considering the case where the finite distributions converge everywhere, rather than just on a dense subset of $\reals_+$.
\begin{lemma}\label{lemma:cond2 preserved under lim}
Let $(\PP_n)_{n\in\nat}$ be probability measures on $(\D,\setsF)$ which satisfy property \ref{item:lemma:ac cond} of Lemma \ref{lemma:finite dist of ac process}. If $\PP_n\rightarrow\PP$ in the sense of finite-dimensional distributions, then $\PP$ also satisfies this property.
\end{lemma}
\begin{proof}
First choose $s<t\in\reals_+$ and real numbers $a$ and $b<c\le d<e$. Also choose times $t_1,t_2,\ldots,t_r\in[0,s]$ and non-negative continuous and bounded functions $u,v:\reals^r\rightarrow\reals$. We let $U,V$ be the $\setsF_s$-measurable random variables
\begin{align*}
&U=u(X_{t_1},X_{t_2},\ldots,X_{t_r}),\\
&V=u(X_{t_1},X_{t_2},\ldots,X_{t_r}).
\end{align*}
 Then, for any real numbers $a_1<a_2$ and $b^\prime<c^\prime\le d^\prime<e^\prime$ inequality (\ref{eqn:lemma:finite dist of ac process:1}) gives
 \begin{equation*}\begin{split}
&\EP{\PP_n}{U1_{\left\{X_s<a_1,\,d^\prime<X_t<e^\prime\right\}}}
 \EP{\PP_n}{V1_{\left\{X_s>a_2,\,b^\prime<X_t<c^\prime\right\}}}\\
={}&
 \EP{\PP_n}{\left(U1_{\left\{X_s<a_1\right\}}\right)1_{\left\{X_s<a_2,\,d^\prime<X_t<e^\prime\right\}}}
 \EP{\PP_n}{V1_{\left\{X_s>a_2,\,b^\prime<X_t<c^\prime\right\}}}\\
\le{}&
 \EP{\PP_n}{\left(U1_{\left\{X_s<a_1\right\}}\right)1_{\left\{X_s<a_2,\,b^\prime<X_t<c^\prime\right\}}}
 \EP{\PP_n}{V1_{\left\{X_s>a_2,\,d^\prime<X_t<e^\prime\right\}}}\\
={}& \EP{\PP_n}{U1_{\left\{X_s<a_1,\,b^\prime<X_t<c^\prime\right\}}}
 \EP{\PP_n}{V1_{\left\{X_s>a_2,\,d^\prime<X_t<e^\prime\right\}}}
 \end{split}\end{equation*}
If we take limits as $n$ goes to infinity and use convergence of the finite-dimensional distributions then this gives
\begin{equation*}\begin{split}
&\EP{\PP}{U1_{\left\{X_s<a_1,\,d^\prime<X_t<e^\prime\right\}}}
\EP{\PP}{V1_{\left\{X_s>a_2,\,b^\prime<X_t<c^\prime\right\}}}\\
\le{}&
\EP{\PP}{U1_{\left\{X_s\le a_1,\,b^\prime\le X_t\le c^\prime\right\}}}
\EP{\PP}{V1_{\left\{X_s\ge a_2,\,d^\prime\le X_t\le e^\prime\right\}}}.
\end{split}\end{equation*}
 Taking limits as $a_1\uparrow a$, $a_2\downarrow a$, $b^\prime\downarrow b$, $c^\prime\uparrow c$, $d^\prime\downarrow d$ and $e^\prime\uparrow e$ gives
\begin{equation*}\begin{split}
&\EP{\PP}{U1_{\left\{X_s<a,\,d<X_t<e\right\}}}
\EP{\PP}{V1_{\left\{X_s>a,\,b<X_t<c\right\}}}\\
\le{}&
\EP{\PP}{U1_{\left\{X_s< a,\,b< X_t< c\right\}}}
\EP{\PP}{V1_{\left\{X_s> a,\,d< X_t< e\right\}}}.
\end{split}\end{equation*}
Note that the set of pairs of random variables $(U,V)$ for which this inequality is true is closed under bounded limits, and under increasing limits. Therefore, it extends to all non-negative and $\setsF_s$-measurable random variables $(U,V)$.
\end{proof}
We now extend this result to the case where convergence is on a dense subset of $\reals_+$.
\begin{corollary}\label{cor:cond2 preserved under dense lim}
Let $(\PP_n)_{n\in\nat}$ be probability measures on $(\D,\setsF)$ which satisfy property \ref{item:lemma:ac cond} of Lemma \ref{lemma:finite dist of ac process}. If $\PP_n\rightarrow\PP$ in the sense of finite-dimensional distributions on a dense subset of $\reals_+$, then $\PP$ also satisfies this property.
\end{corollary}
\begin{proof}
Let $S$ be a dense subset of $\reals_+$ such that $\PP_n\rightarrow\PP$ in the sense of finite-dimensional distributions on $S$.
Then, for every $m\in\nat$ we can find a sequence $(t_{m,k})_{k\in\nat}\in S$ such that $(k-1)/m\le t_{m,k}< k/m$. We define
\begin{align*}
&\theta_m:\reals_+\rightarrow S,\\
&\theta_m(t)=\min\left\{t_{m,k}:k\in\nat,\,t_{m,k}> t\right\}.
\end{align*}
Then, $\theta_m$ is a right-continuous and non-decreasing function, so if we define the process $X^m_t\equiv X_{\theta(t)}$ then it is clear that property \ref{item:lemma:ac cond} of Lemma \ref{lemma:finite dist of ac process} is satisfied if we replace $X$ by $X^m$ under the measure $\PP_n$ (and use the natural filtration generated by $X^m$). Therefore, letting $\QQ_{n,m}$ be the measure on $(\D,\setsF)$ under which $X$ has the same distribution as $X^m$ has under $\PP_n$, then property \ref{item:lemma:ac cond} of Lemma \ref{lemma:finite dist of ac process} is satisfied for the measure $\QQ_{n,m}$. Also, for every $m\in\nat$, let $\QQ_m$ be the measure on $(\D,\setsF)$ under which $X$ has the same distribution as $X^m$ has under $\PP$.

As $\PP_n\rightarrow\PP$ in the sense of finite-dimensional distributions on $S$, then it follows that
\begin{equation}\label{proof:cond2 preserved under dense lim:1}
\QQ_{n,m}\rightarrow\QQ_m
\end{equation}
as $n\rightarrow\infty$, in the sense of finite-dimensional distributions (on all of $\reals_+$).
Also, $\theta_m(t)\ge t$ and $\theta_m(t)\rightarrow t$ as $m\rightarrow \infty$. Therefore, right-continuity of $X_t$ gives $X^m_t\rightarrow X_t$. So,
\begin{equation}\label{proof:cond2 preserved under dense lim:2}
\QQ_{m}\rightarrow\PP
\end{equation}
in the sense of finite-dimensional distributions as $m\rightarrow\infty$. Applying Lemma \ref{lemma:cond2 preserved under lim} to the limit (\ref{proof:cond2 preserved under dense lim:1}) tells us that $\QQ_m$ satisfies property \ref{item:lemma:ac cond} of Lemma \ref{lemma:finite dist of ac process}. Finally, applying Lemma \ref{lemma:finite dist of ac process} to limit (\ref{proof:cond2 preserved under dense lim:2}) shows that $\PP$ also satisfies this property.
\end{proof}

Now, we can finish off the proof of Lemma \ref{lemma:lim of acd is ac}.
\begin{proof}
As $X$ is an almost-continuous diffusion under each of the measures $\PP_n$, the second property of Lemma \ref{lemma:finite dist of ac process} is satisfied. Corollary \ref{cor:cond2 preserved under dense lim} then tells us that $\PP$ also satisfies this property. So, Lemma \ref{lemma:finite dist of ac process} says that $X$ is almost-continuous under $\PP$.
\end{proof}

We shall now give a quick proof of Lemma \ref{lemma:ac connected supp is cont}. First, we will make use of the following result that says that the paths of a process lie inside its marginal support.

\begin{lemma}\label{lemma:process contained in msupp}
Let $X$ be a cadlag real valued process. Then, with probability $1$, we have
\begin{equation*}
\left\{(t,X_t):t\in\reals_+\right\}\subseteq\msupport{X}.
\end{equation*}
Furthermore, if $X$ is continuous in probability then
\begin{equation*}
\left\{(t,X_{t-}):t\in\reals_+\right\}\subseteq\msupport{X}
\end{equation*}
with probability $1$.
\end{lemma}
\begin{proof}
See \citep{Lowther1}, Lemma 4.3 and Lemma 4.4.
\end{proof}

The proof of Lemma \ref{lemma:ac connected supp is cont} follows easily.
\begin{proof}
First, by the statement of the lemma, there exists a countable $S\subseteq\reals_+$ such that $\support{X_t}$ is connected for every $t\in\reals_+\setminus S$. As $X$ is cadlag, there exist stopping times $(\tau_n)_{n\in\nat}$ such that the jump times of $X$ are almost-surely contained in $\cup_{n\in\nat}\slbrack \tau_n\srbrack$ (see \citep{HeWangYan} Theorem 3.32). By Lemma \ref{lemma:process contained in msupp} we have
\begin{equation*}
(\tau_n,X_{\tau_n}), (\tau_n,X_{\tau_n-})\in\msupport{X}
\end{equation*}
whenever $\tau_n<\infty$ (a.s.). Also, by almost-continuity, the second condition of Lemma \ref{lemma:never crosses MSupp cond} is satisfied, and therefore the set
\begin{equation*}
\left\{(\tau_n,x):x\in\reals,\,X_{\tau_n-}<x<X_{\tau_n}\right\}
\end{equation*}
is almost-surely disjoint from the marginal support of $X$, whenever $\tau_n<\infty$.
So, the connected open components of the complement of the set
\begin{equation*}
A=\left\{x\in\reals:(\tau_n,x)\in\msupport{X}\right\}
\end{equation*}
includes the interval $(X_{\tau_n-},X_{\tau_n})$ whenever $\tau_n<\infty$ and $X_{\tau_n-}<X_{\tau_n}$, so $A$ is not connected in this case and we see that $\tau_n\in S$.
Therefore, the continuity in probability of $X$ gives
\begin{equation*}
\Prob{\tau_n<\infty,\,X_{\tau_n-}<X_{\tau_n}}\le\sum_{t\in S}\Prob{\tau_n=t,\,X_{t-}<X_t}=0.
\end{equation*}
Similarly, applying the same argument to $-X$,
\begin{equation*}
\Prob{\tau_n<\infty,\,X_{\tau_n-}>X_{\tau_n}}=0.
\end{equation*}
So $X_{\tau_n-}=X_{\tau_n}$ (a.s.) whenever $\tau_n<\infty$, which shows that $X$ is continuous.
\end{proof}

\section{The Strong Markov Property}
\label{sec:sm}

The aim of this section is to complete the proof of Theorem \ref{thm:limit of acd is acd} by showing that the process $X$ is strong Markov under the measure $\PP$. To do this, we make use of the property that conditional expectations of Lipschitz continuous functions of $X_t$ are themselves Lipschitz continuous.
In this definition we write $f^\prime$ and $g^\prime$ to denote the derivatives $df(x)/dx$ and $dg(x)/dx$ in the measure-theoretic sense, which always exist for Lipschitz continuous functions.

\begin{definition}\label{defn:lipschitz prop}
Let $X$ be any real valued and adapted stochastic process. We shall say that it satisfies the \emph{Lipschitz property} if for all $s<t\in\reals_+$ and every bounded Lipschitz continuous $g:\reals\rightarrow\reals$ with $|g^\prime|\le 1$, there exists a Lipschitz continuous $f:\reals\rightarrow\reals$ with $|f^\prime|\le 1$ and,
\begin{equation*}
f(X_s)=\E{g(X_t)|\setsF_s}.
\end{equation*}
\end{definition}
The reason that we use this property is that it is preserved under taking limits in the sense of finite-dimensional distributions (as we shall see), and is a sufficient condition for the process $X$ to be strong Markov.
\begin{lemma}\label{lemma:lipschitz prop gives sm}
Let $X$ be a cadlag adapted real valued process that satisfies the Lipschitz property. Then, it is strong Markov.
\end{lemma}
\begin{proof}
Choose any $s>0$ and Lipschitz continuous and bounded $g:\reals\rightarrow\reals$. By the Lipschitz property there exists an $f:\halfplane\rightarrow\reals$ such that $f(t,x)$ is Lipschitz continuous in $x$ and
\begin{equation}\label{eqn:pf:lipschitz prop gives sm}
f(\tau,X_\tau)=\E{g(X_{\tau+s})|\setsF_\tau}\ \textrm{ (a.s.)}
\end{equation}
for every $\tau\in\reals_+$. By linearity, this extends to all stopping times $\tau$ that take only finitely many values in $\reals_+$. We shall show that $f(t,x)$ is right-continuous in $t$ on the marginal support of $X$. So, pick any $t\ge 0$ and sequence $t_n\downarrow\downarrow t$.
By the right-continuity of $X$ and uniform continuity of $f(t,x)$ and $g(x)$ in $x$,
\begin{equation*}
\E{g(X_{t+s})|\setsF_{t+}}=\lim_{n\rightarrow\infty} f(t_n,X_{t_n})=\lim_{n\rightarrow\infty} f(t_n,X_t)
\end{equation*}
where convergence is in probability. Taking conditional expectations with respect to $\setsF_t$,
\begin{equation*}
\lim_{n\rightarrow\infty} f(t_n,X_t)=\E{g(X_{t+s})|\setsF_t}=f(t,X_t).
\end{equation*}
By uniform continuity in $x$, this shows that $f(t_n,x)\rightarrow f(t,x)$ for every $x$ in the support of $X_t$ and it follows that $f(t,x)$ is right-continuous in $t$ on the marginal support of $X$.
So Lemma \ref{lemma:process contained in msupp} shows that $f(t,X_t)$ is a right-continuous process and, by taking right limits in $\tau$, (\ref{eqn:pf:lipschitz prop gives sm}) extends to all finite stopping times $\tau$. Finally, the monotone class lemma extends this to all measurable and bounded $g$.
\end{proof}
We now prove the Lipschitz property for the case where the drift term $b(t,x)$ is decreasing in $x$.
\begin{lemma}\label{lemma:acd gives lip}
Let $\PP$ be a probability measure on $(\D,\setsF)$ under which $X$ is an almost-continuous diffusion which decomposes as
\begin{equation*}
X_t=M_t+\int_0^t b(s,X_s)\,ds
\end{equation*}
where $M$ is a local martingale and $b:\halfplane\rightarrow\reals$ is locally integrable such that $b(t,x)$ is decreasing in $x$.

Then, $X$ satisfies the Lipschitz property under $\PP$.
\end{lemma}
\begin{proof}
Fix any $s<t\in\reals_+$ and let $g:\reals\rightarrow\reals$ be bounded and Lipschitz continuous with $|g^\prime|\le 1$. As $X$ is strong Markov, there exists a measurable $h:[0,t]\times\reals\rightarrow\reals$ such that
\begin{equation*}
1_{\{\tau\le t\}}h(\tau,X_\tau)=1_{\{\tau\le t\}}\E{g(X_t)|\setsF_\tau}
\end{equation*}
for every stopping time $\tau$. This follows easily from the strong Markov property (see \citep{Lowther1}, Lemma 2.1).

We now let $(\D^2,\setsF^2,(\setsF^2_t)_{t\in\reals_+},\tilde\PP)$ be the filtered probability space defined by equations (\ref{eqn:D2 F2}), (\ref{eqn:F2t}) and (\ref{eqn:Ptilde}). We also let $Y,Z$ be the stochastic processes on $(\D^2,\setsF^2)$ defined by equation (\ref{eqn:Y Z}).
Then, $Y,Z$ are independent adapted cadlag processes each with the same distribution under $\tilde\PP$ as $X$ has under $\PP$. So they have the decompositions
\begin{equation}\label{eqn:proof:acd gives lip:1}\begin{split}
&Y_u=M^1_u+\int_0^u b(v,Y_v)\,dv,\\
& Z_u=M^2_u+\int_0^u b(v,Z_v)\,dv
\end{split}\end{equation}
for local martingales $M^1,M^2$. Furthermore, $Y,Z$ will also be strong Markov and satisfy
\begin{equation}\label{eqn:proof:acd gives lip:2}\begin{split}
&1_{\{\tau\le t\}}h(\tau,Y_\tau)=1_{\{\tau\le t\}}\Et{g(Y_t)|\setsF_\tau},\\
&1_{\{\tau\le t\}}h(\tau,Z_\tau)=1_{\{\tau\le t\}}\Et{g(Z_t)|\setsF_\tau}
\end{split}\end{equation}
for all $\setsF^2_\cdot$-stopping times $\tau$. This follows quite easily from the definitions of $Y,Z$ (see \citep{Lowther1}, Lemma 2.2).

Now let $\tau$ be the stopping time
\begin{equation*}
\tau=\inf\left\{u\in[s,\infty):Y_u\ge Z_u\right\}.
\end{equation*}
We note that $\{\tau>s\}=\{Y_s<Z_s\}$. Then equation (\ref{eqn:proof:acd gives lip:2}) gives
\begin{equation*}\begin{split}
1_{\{\tau>s\}}\left(h(s,Z_s)-h(s,Y_s)\right)
={}&1_{\{\tau>s\}}\Et{g(Z_t)-g(Y_t)|\setsF^2_s}\\
={}&\Et{1_{\{\tau\ge t\}}\left(g(Z_t)-g(Y_t)\right)|\setsF^2_s}\\
&+\Et{1_{\{t>\tau>s\}}\left(h(\tau,Z_\tau)-h(\tau,Y_\tau)\right)|\setsF^2_s}.
\end{split}\end{equation*}
The almost-continuity of $X$ gives $Y_\tau=Z_\tau$ whenever $t>\tau>s$, so
\begin{equation}\label{eqn:proof:acd gives lip:3}\begin{split}
1_{\{\tau>s\}}\left(h(s,Z_s)-h(s,Y_s)\right)
&=\Et{1_{\{\tau\ge t\}}\left(g(Z_t)-g(Y_t)\right)|\setsF^2_s}\\
&=\Et{1_{\{\tau>s\}}\left(g(Z_{t\wedge\tau})-g(Y_{t\wedge\tau})\right)|\setsF^2_s}\\
&\le\Et{1_{\{\tau>s\}}\left(Z_{t\wedge\tau}-Y_{t\wedge\tau}\right)|\setsF^2_s}.
\end{split}\end{equation}
Here, we made use of the condition that $|g^\prime|\le 1$.

Let $N$ be the local martingale
\begin{equation}\label{eqn:proof:acd gives lip:4}
N_u=1_{\{\tau>s\}}\left(M^2_{u\wedge\tau}-M^1_{u\wedge\tau}+\int_0^s\left(b(u,Z_u)-b(u,Y_u)\right)\,du\right)
\end{equation}
defined over $u\ge s$.
Then for every $u\ge s$, the condition that $b(u,x)$ is decreasing in $x$ gives
\begin{equation}\label{eqn:proof:acd gives lip:5}
N_u -1_{\{\tau>s\}}\left(Z_{u\wedge\tau}-Y_{u\wedge\tau}\right)
=\int_s^{u\wedge\tau}\left(b(v,Y_v)-b(v,Z_v)\right)\,dv\ge 0.
\end{equation}
So, substituting into inequality (\ref{eqn:proof:acd gives lip:3}),
\begin{equation}\label{eqn:proof:acd gives lip:6}
1_{\{\tau>s\}}\left(h(s,Z_s)-h(s,Y_s)\right)\le 1_{\{\tau>s\}}\left(Z_s-Y_s\right)+\E{N_t-N_s|\setsF^2_s}.
\end{equation}
Inequality (\ref{eqn:proof:acd gives lip:5}) shows that $N_u$ is a non-negative local martingale and is therefore a supermartingale. So $\E{N_t|\setsF_s}\le N_s$ and inequality (\ref{eqn:proof:acd gives lip:6}) gives
\begin{equation*}
h(s,Z_s)-h(s,Y_s)\le Z_s-Y_s
\end{equation*}
whenever $Y_s<Z_s$ (almost surely). Replacing $h$ by $-h$ in the above argument will also give the above inequality with $Y$ and $Z$ interchanged on the left hand side. Furthermore, the inequality still holds if we interchange $Y$ and $Z$ on both sides (by symmetry). So
\begin{equation*}
|h(s,Z_s)-h(s,Y_s)|\le| Z_s-Y_s|
\end{equation*}
(almost surely). As $Y_s,Z_s$ are independent and each have the same distribution as $X_s$, this shows that $h(s,\cdot)$ is Lipschitz continuous in an almost-sure sense. That is, there is a measurable $A\subseteq\reals$ such that $\Prob{X_s\in A}=1$ and such that
\begin{equation*}
|h(s,x)-h(s,y)|\le |x-y|
\end{equation*}
for every $x,y\in A$. Therefore we can define $f:\reals\rightarrow\reals$ by $f(x)=h(s,x)$ for all $x\in A$. By uniform continuity, this extends uniquely to the closure $\bar A$ of $A$ such that
\begin{equation*}
|f(x)-f(y)|\le|x-y|
\end{equation*}
for every $x,y\in\bar A$. Then we can extend $f$ linearly across each open interval in the complement of $\bar A$ so that $f$ is Lipschitz continuous with $|f^\prime|\le 1$. Finally, $f(X_s)=h(s,X_s)$ whenever $X_s\in A$ so,
\begin{equation*}
f(X_s)=h(s,X_s)=\E{g(X_t)|\setsF_s}.\qedhere
\end{equation*}
\end{proof}

We now extend this result to cover the case where $b(t,x)$ just satisfies the Lipschitz condition required by Theorem \ref{thm:limit of acd is acd}.
\begin{corollary}\label{cor:acd with K gives lip}
Let $X$ be an almost-continuous diffusion that decomposes as
\begin{equation*}
X_t=M_t+\int_0^t b(s,X_s)\,ds
\end{equation*}
where $M$ is a local martingale, $b:\halfplane\rightarrow\reals$ is locally integrable and such that there exists a $K\in\reals$ satisfying
\begin{equation*}
b(t,y)-b(t,x)\le K(y-x)
\end{equation*}
for every $t\in\reals_+$ and $x<y\in\reals$.

Then, $e^{-Kt}X_t$ satisfies the Lipschitz property.
\end{corollary}
\begin{proof}
If we set $Y_t=e^{-Kt}X_t$ then $Y$ is an almost-continuous diffusion and integration by parts gives
\begin{equation*}
Y_t = N_t+\int_0^t c(s,Y_s)\,ds
\end{equation*}
where
\begin{align*}
&N_t = X_0+\int_0^t e^{-Ks}\,dM_s,\\
&c(t,y)=e^{-Kt}b(t,e^{Kt}y)-Ky.
\end{align*}
As $N$ is a local martingale and $c(t,y)$ is decreasing in $y$, the result follows from Lemma \ref{lemma:acd gives lip}.
\end{proof}

In order to show that the limit in Theorem \ref{thm:limit of acd is acd} is strong Markov, we shall show that $e^{-Kt}X_t$ satisfies the Lipschitz property. This works because this property is preserved under taking limits in the sense of finite-dimensional distributions.

\begin{lemma}\label{lemma:lim everywhere of lip prop gives lip}
Let $(\PP_n)_{n\in\nat}$ and $\PP$ be probability measures on $(\D,\setsF)$ such that $\PP_n\rightarrow\PP$ in the sense of finite-dimensional distributions.

If the Lipschitz property for $X$ is satisfied under each $\PP_n$ then it is also satisfied under $\PP$.
\end{lemma}
\begin{proof}
Fix any $s<t\in\reals_+$ and any bounded and Lipschitz continuous $g:\reals\rightarrow\reals$ such that $|g|\le K$ and $|g^\prime|\le 1$. Then, by the Lipschitz property for $X$ under $\PP_n$, there exist Lipschitz continuous functions $f_n:\reals\rightarrow\reals$ such that $|f_n|\le K$, $|f^\prime_n|\le 1$ and
\begin{equation*}
f_n(X_s)=\EP{\PP_n}{g(X_t)|\setsF_s}.
\end{equation*}
Now let $S\subseteq\reals$ be the support of $X_s$ under $\PP$. We shall show that $f_n$ converges pointwise on $S$ as $n\rightarrow\infty$.
So, pick any $x\in S$ and any $\epsilon>0$. Let $\theta:\reals\rightarrow\reals$ be any continuous and non-negative function with support contained in $[x-\epsilon,x+\epsilon]$ such that $\theta(x)>0$.
As $x\in S$ we have
\begin{equation*}
\EP{\PP}{\theta(X_s)}=\delta>0.
\end{equation*}
We use the following simple identity
\begin{equation*}\begin{split}
\delta\left(f_n(x)-f_m(x)\right)={}&\EP{\PP_n}{\theta(X_s)f_n(x)}-\EP{\PP_m}{\theta(X_s)f_m(x)}\\
&-\left(\EP{\PP_n}{\theta(X_s)}-\EP{\PP}{\theta(X_s)}\right)f_n(x)\\
&+\left(\EP{\PP_m}{\theta(X_s)}-\EP{\PP}{\theta(X_s)}\right)f_m(x)
\end{split}\end{equation*}
Convergence of the distribution of $X_s$ under $\PP_n$ to its distribution under $\PP$ (as $n\rightarrow\infty$) tells us that the final two terms on the right hand side of this inequality vanish as we take limits, so
\begin{equation}\label{eqn:proof:lim everywhere of lip prop gives lip:1}\begin{split}
&\delta\limsup_{m,n\rightarrow\infty}\left|f_n(x)-f_m(x)\right|\\
&\le\limsup_{m,n\rightarrow\infty}\left|\EP{\PP_n}{\theta(X_s)f_n(x)}-\EP{\PP_m}{\theta(X_s)f_m(x)}\right|.
\end{split}\end{equation}
Also, Lipschitz continuity of $f_n$ on the interval $[x-\epsilon,x+\epsilon]$ gives
\begin{equation*}
\left|\theta(X_s)\left(f_n(X_s)-f_n(x)\right)\right|\le\epsilon\theta(X_s)
\end{equation*}
so
\begin{equation*}\begin{split}
&\left|\EP{\PP_n}{\theta(X_s)f_n(x)}-\EP{\PP_m}{\theta(X_s)f_m(x)}\right|\\
\le{}&
\left|\EP{\PP_n}{\theta(X_s)f_n(X_s)}-\EP{\PP_m}{\theta(X_s)f_m(X_s)}\right|\\
&+ \epsilon\EP{\PP_n}{\theta(X_s)}+\epsilon\EP{\PP_m}{\theta(X_s)}\\
={}&  \left|\EP{\PP_n}{\theta(X_s)g(X_t)}-\EP{\PP_m}{\theta(X_s)g(X_t)}\right|\\
&+ \epsilon\EP{\PP_n}{\theta(X_s)}+\epsilon\EP{\PP_m}{\theta(X_s)}.
\end{split}\end{equation*}
If we take limits as $m,n\rightarrow\infty$ then the convergence of the finite-dimensional distributions of $\PP_n$ and $\PP_m$ to $\PP$ shows that the right hand side of this inequality converges to $2\epsilon\delta$,
\begin{equation*}
\limsup_{m,n\rightarrow\infty}\left|\EP{\PP_n}{\theta(X_s)f_n(x)}-\EP{\PP_m}{\theta(X_s)f_m(x)}\right|\le2\epsilon\delta.
\end{equation*}
Substituting this into inequality (\ref{eqn:proof:lim everywhere of lip prop gives lip:1}) gives
\begin{equation*}
\limsup_{m,n\rightarrow\infty}\left|f_n(x)-f_m(x)\right|\le 2\epsilon.
\end{equation*}
As this is true for every $\epsilon>0$, the sequence $f_n(x)$ is Cauchy, and therefore converges as $n$ goes to infinity. So we can define $f:S\rightarrow\reals$ by $f(x)=\lim_{n\rightarrow\infty}f_n(x)$. Then
\begin{equation}\label{eqn:proof:lim everywhere of lip prop gives lip:2}
|f(x)-f(y)|=\lim_{n\rightarrow\infty}|f_n(x)-f_n(y)|\le |x-y|
\end{equation}
for every $x,y\in S$. So $f$ is Lipschitz continuous on $S$. By interpolating and extrapolating $f$ linearly across the connected open components of $\reals\setminus S$, we can extend it to a function $f:\reals\rightarrow\reals$ such that inequality (\ref{eqn:proof:lim everywhere of lip prop gives lip:2}) is satisfied for all $x,y\in\reals$. So, $f$ is Lipschitz continuous with $|f^\prime|\le 1$. Also, we can choose $f$ such that $|f|\le K$.
To complete the proof of the lemma, it only remains to show that $f(X_s)=\EP{\PP}{g(X_t)|\setsF_s}$.

Now set
\begin{equation*}
h(x)=\limsup_{n\rightarrow\infty}|f_n(x)-f(x)|
\end{equation*}
so that $h$ vanishes on $S$, and is a bounded Lipschitz continuous function satisfying $|h|\le 2K$ and $|h^\prime|\le2$. By uniform continuity of the functions $f_n$, the convergence is uniform on bounded subsets of $\reals$. That is, for every $A>0$,
\begin{equation*}
|f_n(x)-f(x)|\le h(x) + 1/A
\end{equation*}
for all large $n$ and $|x|\le A$. So,
\begin{equation*}\begin{split}
&\limsup_{n\rightarrow\infty}\EP{\PP_n}{|f_n(X_s)-f(X_s)|}\\
\le{}& \limsup_{n\rightarrow\infty}\EP{\PP_n}{h(X_s)}+1/A+2K\limsup_{n\rightarrow\infty}\PP_n\left(|X_s|>A\right)\\
\le{}& \EP{\PP}{h(X_s)}+1/A+2K\PP\left(|X_s|\ge A\right)\\
={}& 1/A +2K\PP\left(|X_s|\ge A\right).
\end{split}\end{equation*}
Letting $A$ increase to infinity gives
\begin{equation}\label{eqn:proof:lim everywhere of lip prop gives lip:3}
\limsup_{n\rightarrow\infty}\EP{\PP_n}{|f_n(X_s)-f(X_s)|}=0.
\end{equation}

Finally, choose any finite set of times $t_1,t_2,\ldots,t_d\in[0,s]$, choose any bounded and continuous $u:\reals^d\rightarrow\reals$, and let $U$ be the $\setsF_s$-measurable random variable
\begin{equation*}
U=u\left(X_{t_1},X_{t_2},\ldots,X_{t_d}\right)
\end{equation*}
Then, letting $L$ be an upper bound for $|u|$, we can use the equality $f_n(X_s)=\EP{\PP_n}{g(X_t)|\setsF_s}$ and equation (\ref{eqn:proof:lim everywhere of lip prop gives lip:3}) to get
\begin{equation*}\begin{split}
\left|\EP{\PP}{U\left(f(X_s)-g(X_t)\right)}\right| ={}& \lim_{n\rightarrow\infty}\left|\EP{\PP_n}{U(f(X_s)-g(X_t))}\right|\\
\le{}&\limsup_{n\rightarrow\infty}\left|\EP{\PP_n}{U(f_n(X_s)-g(X_t))}\right|\\
&+L\limsup_{n\rightarrow\infty}\EP{\PP_n}{\left|f_n(X_s)-f(X_s)\right|}\\
={}&0.
\end{split}\end{equation*}
Therefore, $\EP{\PP}{Uf(X_s)}=\EP{\PP}{Ug(X_t)}$. By the monotone class lemma this extends to all bounded and $\setsF_s$-measurable $U$, so $f(X_s)=\EP{\PP}{g(X_t)|\setsF_s}$.
\end{proof}

This result is for convergence everywhere of the finite-dimensional distributions. It is easy to extend it to only require convergence on a dense subset of $\reals_+$.

\begin{corollary}\label{cor:lim of lip prop gives lip}
Let $(\PP_n)_{n\in\nat}$ and $\PP$ be probability measures on $(\D,\setsF)$ such that $\PP_n\rightarrow\PP$ in the sense of finite-dimensional distributions on a dense subset of $\reals_+$.

If the Lipschitz property for $X$ is satisfied under each $\PP_n$ then it is also satisfied under $\PP$.
\end{corollary}
\begin{proof}
We imply this result from Lemma \ref{lemma:lim everywhere of lip prop gives lip} in the same way that Corollary \ref{cor:cond2 preserved under dense lim} followed from Lemma \ref{lemma:cond2 preserved under lim}. So, let $\QQ_{n,m}$ and $\QQ_m$ be the probability measures on $(\D,\setsF)$ defined in the proof of Corollary \ref{cor:cond2 preserved under dense lim}.

It is clear that the Lipschitz property for $X$ under $\PP_n$ implies that it also satisfies the Lipschitz property under $\QQ_{n,m}$. Then Lemma \ref{lemma:lim everywhere of lip prop gives lip} applied to the limit (\ref{proof:cond2 preserved under dense lim:1}) says that $X$ satisfies the Lipschitz property under $\QQ_m$. Applying Lemma \ref{lemma:lim everywhere of lip prop gives lip} to the limit (\ref{proof:cond2 preserved under dense lim:2}) gives the result.
\end{proof}

We finally prove Theorem \ref{thm:limit of acd is acd}.
\begin{proof}
First, Lemma \ref{lemma:lim of acd is ac} says that $X$ is almost-continuous under the measure $\PP$.
Also, Corollary \ref{cor:acd with K gives lip} says that $e^{-Kt}X_t$ satisfies the Lipschitz property under $\PP_n$, so by Corollary \ref{cor:lim of lip prop gives lip} it also satisfies the Lipschitz property under $\PP$. Lemma \ref{lemma:lipschitz prop gives sm} then says that $e^{-Kt}X_t$ is a strong Markov process under $\PP$, and therefore $X$ is also a strong Markov process.
\end{proof}

\appendix

\bibliography{limits.bbl}
\bibliographystyle{acmtrans-ims}

\end{document}